\documentclass[12pt,reqno]{amsart}
\usepackage{amssymb,amsmath,amsthm,amscd,latexsym,amsfonts}

\usepackage{xy}
\xyoption{all}
\usepackage{mathtools}
\usepackage[T1]{fontenc}
\usepackage{graphicx}
\usepackage{dsfont}
\usepackage{cite}
\usepackage{bbm}
\usepackage{color}
\usepackage{enumitem}

\usepackage{orcidlink}
\DeclareMathAlphabet{\mathpzc}{OT1}{pzc}{m}{it}

\usepackage[a4paper,top=3cm, bottom=3cm, left=3cm, right=3cm]{geometry}

\newtheorem{lemma}[subsection]{Lemma}
\newtheorem{proposition}[subsection]{Proposition}
\newtheorem{corollary}[subsection]{Corollary}

\theoremstyle{definition}
\newtheorem{remark}[subsection]{Remark}
\newtheorem{definition}[subsection]{Definition}

\numberwithin{equation}{section}

\theoremstyle{definition}

\def\de{\delta}

\def\lam{\lambda}

\DeclareMathOperator{\Ker}{\mathsf {Ker}}

\DeclareMathOperator{\Img}{\mathsf {Im}}

\DeclareMathOperator{\ab}{ab}

\def\lim{\operatorname{lim}}

\begin{document}
\title[Schur's theorem and its relation with the non-abelian tensor product]{Schur's theorem and its relation to the closure properties of the non-abelian tensor product}

\author[G.~Donadze]{Guram Donadze \orcidlink{0000-0003-3536-9599}}
\address{[G.~Donadze] Departamento de Matematica, Universidade de Brasilia, Brasilia-DF,
	70910-900 Brazil and Vladimir Chavchanidze Institute of Cybernetics of
	the Georgian Technical University, Sandro Euli str. 5, Tbilisi 0186,
	Georgia}
	\email{gdonad@gmail.com}
	
	\author[M. Ladra]{Manuel Ladra \orcidlink{0000-0002-0543-4508}}
	\address{[M. Ladra] Department of Matem\'aticas, University of Santiago de Compostela, 15782 Santiago de Compostela, Spain}
	\email{manuel.ladra@usc.es}
	
\author[P. P\'aez-Guill\'an]{Pilar P\'aez-Guill\'an \orcidlink{0000-0003-2761-7505}}
	\address{[P. P\'aez-Guill\'an] Department of Matem\'aticas, University of Santiago de Compostela, 15782 Santiago de Compostela, Spain}
	\email{pilar.paez@usc.es}

\begin{abstract} We show that the Schur multiplier of a Noetherian group need not be finitely generated.
We prove that the non-abelian tensor product of a polycyclic (resp. polycyclic-by-finite) group
and a Noetherian group, is a polycyclic (resp. polycyclic-by-finite) group. We also prove new versions of Schur's theorem.

\end{abstract}



\maketitle

\section{Introduction}\label{Intr}

One of the most famous theorems of Schur states that the commutator subgroup of a central-by-finite group is finite. This theorem has been generalized
for various classes of groups; for instance, it has been shown that the commutator subgroup of a central-by-finite $p$-group (or a polycyclic-by-finite group)
is a finite $p$-group (or a polycyclic-by-finite group). In this manuscript we present an apparently new approach to prove Schur's theorem and
its generalizations. Our proof reveals a connection between this theorem and the closure properties of the non-abelian tensor product of groups.
One of the objectives of this paper is to study one of these properties for the class of Noetherian groups.
 Brown and  Loday introduced the non-abelian tensor product $G\otimes H$ for a pair of groups $G$ and $H$ in \cite{BroLod84} and \cite{BroLod87}
  in the context of an application in homotopy theory.
The achieved connection
helps to establish that the Schur multiplier of a Noetherian group is not always finitely generated, which seems to be a rather unexpected result.

If $G$ and $H$ belong to the class $\mathfrak{C}$, does $G\otimes H$ belong to the class $\mathfrak{C}$?
This question has been considered by many authors. Moravec \cite{Mo07} proves that the class of polycyclic groups  is closed under
the formation of the non-abelian tensor product.
 A group is polycyclic if and only if it is solvable and Noetherian. In \cite{DLT}  it is proved that if $G$ is a finite group and $H$ is Noetherian,
  then $G\otimes H$ is also Noetherian. Therefore, the following question is very natural and was posed in \cite{DLT}.

\

{\bf QUESTION.} Let $G$ and $H$ be Noetherian groups acting on each other compatibly. Is $G\otimes H$ a Noetherian group?

\

In the same paper it is shown that for a given Noetherian group $G$, the tensor square $G\otimes G$ is Noetherian if and only if
the Schur multiplier of $G$ is a finitely generated group. We base on Ol'shanski\u\i's result \cite{Olsh} to prove that, in general, the Schur multiplier of a
Noetherian group is not finitely generated, as it was already pointed out. We also generalize the main result of \cite{Mo07} and \cite[Theorem 5.11]{DLT},
proving that if $G$ is a polycyclic (resp. polycyclic-by-finite) group and $H$ is Noetherian, then $G\otimes H$ is a polycyclic (resp. polycyclic-by-finite) group.

The paper is organized as follows. In Section~\ref{S:noab} we provide
some basic results on the non-abelian tensor product and prove  that if $G$ is a polycyclic (resp. polycyclic-by-finite) group and $H$ is Noetherian, then $G\otimes H$ is a polycyclic (resp. polycyclic-by-finite) group.
In Section~\ref{S:schur} we give generalizations of Schur's theorem for various classes of
groups, especially for the super-solvable and the finitely
generated; we also show that the veracity of Schur's theorem for Noetherian groups is related to the ``finiteness'' of the Schur multiplier of the same group. Finally, we point out that the Schur multiplier of a Noetherian group need not be finitely generated.

\section{Non-abelian tensor product of Noetherian groups}\label{S:noab}
The non-abelian tensor product of groups was defined by Brown and Loday \cite{BroLod84, BroLod87} for a pair of groups
acting on each other. We write conjugation on the left, so $^gg'=gg'g^{-1}$ for $g,g'\in G$
and $^gg'g'^{-1}=[g,g']$ for the commutator of $g$ and $g'$.

\begin{definition}\label{D:2.1}
Let $G$ and $H$ be groups acting on themselves by conjugation, each of them acting on the other as well. The mutual actions are said to be \emph{compatible} if
\[
  ^{^h g}h'=\; ^{hgh^{-1}}h' \quad \text{ and }  \quad ^{^g h}g'=\ ^{ghg^{-1}}g', \ \text{for all} \ g,g'\in G \ \text{and} \  h,h'\in H.
\]
\end{definition}

\begin{definition}
If $G$ and $H$ are groups that act compatibly on each other, then the \emph{non-abelian tensor product} $G\otimes H$
is the group generated by the symbols $g\otimes h$ for $g\in G$ and $h\in H$ with relations
\begin{align*}
gg'\otimes h & =(^gg'\otimes \,^gh)(g\otimes h), \\
g\otimes hh' & =(g\otimes h)(^hg\otimes \,^hh'),
\end{align*}
for all $g,g'\in G$ and $h,h'\in H$.
\end{definition}

Given two groups $G$ and $H$ with mutual compatible actions, there is another action of $G$ on $G\otimes H$ defined by
$^g(g'\otimes h')=\:^gg'\otimes \:^gh'$ for each $g, g'\in G$ and $h'\in H$. Moreover, the map $\phi \colon G\otimes H
\to G$, given by $g\otimes h\mapsto g^hg^{-1}$ for each $g\in G$ and $h\in H$, is a crossed module (see  \cite[Proposition 2.3]{BroLod87}).

The special case in which $G=H$, and all actions are given by conjugation, is called the \emph{tensor square} $G\otimes G$.

Let $G$ and $H$ be groups acting on each other compatibly. In \cite{Vi}, the so-called \emph{derivative} of $G$ by $H$ was introduced; it was defined as $D_H(G) =  \langle g \, ^hg^{-1} \ \mid g \in G,  \: h\in H\rangle\subset G$.
Aiming to enhance the understanding of this paper, we recall a result  proved in \cite{DLT}.

\begin{proposition}\label{paper5}
Let $G$ and $H$ be groups acting on each other compatibly. If $G$ and $H$ are finitely generated, then $G\otimes H$ is also finitely generated if and only if so are
$D_H(G)$ and $D_G(H)$.
\end{proposition}

The following useful result is also proved in \cite{DLT}.

\begin{proposition}\label{paper55} Let $G$ be a Noetherian group. Then $G \otimes G$ is Noetherian if and only if the Schur multiplier of $G$ is finitely generated.
\end{proposition}

In Section~\ref{S:schur} we will show that the Schur multiplier of Noetherian groups is not always finitely generated. The next lemma
helps to establish some closure properties of the non-abelian tensor product of Noetherian groups.

\begin{lemma} \label{lemma1} Let $G$ be a group and $H$ be a normal subgroup of $G$. Suppose that $G$ and $H$ act on each other by conjugation.
Let $i \colon H\otimes H \to G\otimes H$ be the homomorphism given by $h\otimes h' \mapsto h\otimes h'$ for each $h, h' \in H$.
Then $i(H\otimes H)$ is a normal subgroup of $G\otimes H$ and $\frac{G\otimes H}{i(H\otimes H)}$ is abelian.
\end{lemma}
\begin{proof}
For each $ g, g' \in G$ and $h,  h'\in H$, the following relation verifies in $G\otimes H$ (see \cite{BJR87}):
\begin{align} \label{ident1}
(g '\otimes h') (g\otimes h)(g'\otimes h')^{-1} = \: ^{[g', h']}g \otimes \: ^{[g', h']}h.
\end{align}
This implies that $i(H\otimes H)$ is a normal subgroup of $G\otimes H$. To show the abelian character of
 $\frac{G\otimes H}{i(H\otimes H)}$, we will use another identity proved in \cite{BJR87}:
\[
g\otimes [h, h'] =\; ^{g} (h\otimes h') (h\otimes h') ^{-1},
\]
for each $h, h'\in H$ and $g\in G$. It holds that
\[
g\otimes \, ^{h'} h = g\otimes  h [h^{-1}, h'] = ( g\otimes  h )^h( g\otimes  [h^{-1}, h']);
\]
combining these two identities, we get:
\begin{align} \label{ident2}
(g\otimes \,^{h'} h) \:  i(H\otimes H) =(g\otimes  h) \: i(H\otimes H) ,
\end{align}
for each $g\in G$ and $h, h'\in H$. Moreover,
\[
^{h'}g\otimes  h = g[g^{-1}, h']\otimes  h = \: ^{h'}([g^{-1}, h']\otimes  h)(g\otimes  h),
\]
hence
\begin{align} \label{ident3}
(^{h'}g\otimes  h) \:  i(H\otimes H) =(g\otimes  h) \: i(H\otimes H) ,
\end{align}
for each $g\in G$ and $h, h'\in H$. Finally, using \eqref{ident1}, \eqref{ident2} and \eqref{ident3}, we obtain
\[
(g'\otimes h')(g\otimes h)(g'\otimes h')^{-1} \: i(H\otimes H) =(g\otimes h)  \: i(H\otimes H) ,
\]
for each $g, g'\in G$ and $h, h'\in H$, and therefore $\frac{G\otimes H}{i(H\otimes H)}$ is an abelian group.
\end{proof}

\begin{corollary} \label{corollary1}  Let $G$ be a Noetherian group and $H$ be a normal subgroup of $G$. Suppose that
$G$ and $H$ act on each other by conjugation. If the Schur multiplier of $H$ is a finitely generated group, then $G\otimes H$
is Noetherian.
\end{corollary}
\begin{proof}  By Proposition~\ref{paper55}, $H\otimes H$ is a Noetherian group. Moreover, Proposition~\ref{paper5}
and Lemma~\ref{lemma1} imply that $\frac{G\otimes H}{i(H\otimes H)}$ is a finitely generated abelian group,
hence $\frac{G\otimes H}{i(H\otimes H)}$ is Noetherian. Thus, from the exact sequence
\[
H\otimes H \to G\otimes H \to \frac{G\otimes H}{i(H\otimes H)} \to 1
\]
we deduce that $G\otimes H$ is Noetherian.
\end{proof}

\begin{lemma} \label{lemma2} Suppose that $G$ and $H$ are normal subgroups of some group which act on each other by conjugation.
If both $G$ and $H$ are Noetherian, and the Schur multiplier of $G\cap H$ is finitely generated, then $G\otimes H$
is also a Noetherian group.
\end{lemma}
\begin{proof} Given two groups $G/G\cap H$ and $H/G\cap H$, we define the mutual actions induced by conjugation, which are trivial
because $[G, H]\subset G\cap H$. We obtain the following exact sequence:
\[
(G\cap H)\otimes H \:\times \: G \otimes (G\cap H) \to G\otimes H \to \frac{G}{G\cap H}\otimes  \frac{H}{G\cap H}\to 1.
\]
Since $G/G\cap H$ and $H/G\cap H$ act on each other trivially,
\[
\frac{G}{G\cap H}\otimes  \frac{H}{G\cap H} = \Bigg(\frac{G}{G\cap H}\Bigg)_{\ab}\otimes
\Bigg(\frac{H}{G\cap H}\Bigg)_{\ab},
\]
where $G_{\ab}$ denotes the abelianization of $G$.

Therefore, $\frac{G}{G\cap H}\otimes  \frac{H}{G\cap H}$ is a finitely generated abelian group. Moreover, by Corollary~\ref{corollary1}
both $(G\cap H)\otimes H$ and  $G \otimes (G\cap H)$ are Noetherian groups. The above exact sequence completes
the proof.
\end{proof}

Given two groups $G$ and $H$ with mutual compatible actions, let $(G, H)$ denote the normal subgroup of the semidirect product
$G \rtimes H$ generated by the elements $(g^hg^{-1}, h^gh^{-1})$ for all $g\in G$ and $h\in H$. Set
$G\circ H = \frac{G \rtimes H}{(G, H)}$. There are actions of $G\circ H$ on $G$ and on $H$ given by $^{(g, h)}g' =
\:^{g}(^hg')$ and  $^{(g, h)}h' =\:^{g}(^hh')$ for all $g, g' \in G$ and $h, h' \in H$, and the natural homomorphisms
$\mu \colon G \to G\circ H$ and $\nu \colon H \to G\circ H$ together with these actions are crossed modules (see \cite{Ell87f}).
Hence, $\Ker \mu$ and $\Ker \nu$ are abelian groups acting trivially on $H$ and on $G$, respectively. We have the following

\begin{lemma} \label{lemma3} If $G$ and $H$ are Noetherian groups and the Schur multiplier of $\mu (G) \cap \nu (H)$ is finitely generated, then $G\otimes H$ is Noetherian.
\end{lemma}
\begin{proof} Since $\Ker \mu$ is an abelian group  acting trivially on $H$, we infer from \cite{Gu88} that
 $\Ker \mu\otimes H = \Ker \mu\otimes_H I(H)$. The latter is a finitely generated abelian group, since $\Ker\mu$ is too and, additionally, $I(H)$ is a finitely generated $H$-module.  For the same reason, $G\otimes \Ker\nu$ is a finitely generated
abelian group. Moreover, by Lemma~\ref{lemma2} $\mu (G)\otimes \nu(H)$ is Noetherian. Thus, the exact sequence
\[
G\otimes \Ker\nu \: \times \:\Ker \mu \otimes H \to G\otimes H \to \mu (G)\otimes \nu(H)\to 1
\]
implies that $G\otimes H$ is a Noetherian group too.
\end{proof}

\begin{proposition}  If  $G$ is a polycyclic (resp. polycyclic-by-finite) group and $H$ is Noetherian, then $G\otimes H$ is a
polycyclic (resp. polycyclic-by-finite) group.
\end{proposition}
\begin{proof}  Being a homomorphic image of a polycyclic (resp. polycyclic-by-finite) group, the group $\mu (G)$ inherits its structure; consequently, $\mu (G)\cap \nu (H)$ is also a polycyclic (resp. polycyclic-by-finite) group, due to its condition of subgroup of $\mu (G)$.
Therefore, the Schur multiplier of $\mu (G)\cap \nu (H)$
is finitely generated, and Lemma~\ref{lemma3} ensures that $G\otimes H$ is Noetherian. Now, consider the homomorphism
$\phi \colon G\otimes H \to G$ defined in the Introduction. Since $\phi$ is a crossed module,
$\Ker \phi$ is contained  in the center of $G\otimes H$; furthermore, it is a finitely generated abelian group, because $G\otimes H$ is Noetherian. Moreover, $\Img \phi$ is a polycyclic (resp. polycyclic-by-finite) group, and consequently
$G\otimes H$ is also polycyclic (resp. polycyclic-by-finite).
\end{proof}

\begin{corollary}
Let $G$ be a simple Noetherian group and $H$ be Noetherian. If the Schur multiplier of $G$ is finitely generated,
then $G\otimes H$ is Noetherian.
\end{corollary}
\begin{proof} Since $\mu \colon G \to G\circ H$ and $\nu \colon H \to G\circ H$ are crossed modules, $\mu (G)$ and $\nu (H)$
are normal subgroups of $G\circ H$, as well as $\mu (G)\cap \nu (H)$.
On the one hand, $G$ is a simple group, so $\mu (G)$ is either trivial, or it is isomorphic to $G$. On the other hand, $\mu (G)\cap \nu (H)$ is a normal
subgroup of $\mu (G)$; using the same argument as before, it is clear that $\mu (G)\cap \nu (H)$ is either a trivial group, or it is isomorphic to $G$.
Hence, the Schur multiplier of $\mu (G)\cap \nu (H)$ is either trivial, or finitely generated.
\end{proof}

\begin{remark}
Note that if $G$ is a finite group and $H$ is Noetherian, then the non-abelian tensor product $G\otimes H$ is not
always finite.  For instance, assume that $G=\mathbb{Z}/2\mathbb{Z}=\langle t \mid t+t=0\rangle$ and
$H=\mathbb{Z}$. Defining an action of $G$ on $\mathbb{Z}$ by $^tn=-n$ for each
$n\in \mathbb{Z}$, and assuming that $\mathbb{Z}$ acts trivially on $G$, we obtain that $\mathbb{Z}$ and $G$
act on each other compatibly, but $G\otimes\mathbb{Z}$ is isomorphic to $\mathbb{Z}$.
\end{remark}

\begin{remark}
We would like to mention here that the finiteness problem of the Schur multiplier of Noetherian groups can be studied from a topological approach. In particular,
given a group $G$, let $\pi_2^S(K(G, 1))$ denote the second stable homotopy group of the Eilenberg-MacLane space
$K(G, 1)$. The following exact sequence of groups is provided by \cite[Theorem 4.7]{ADST}:
\[
1 \to \frac{G_{\ab}}{2G_{\ab}}\to \pi_2^S(K(G, 1)) \to H_2(G)\to 1,
\]
where $2G_{\ab}$ is a subgroup of $G_{\ab}$ generated by all elements
of the form $2x$, $x\in G_{\ab}$. If, additionally, $G$ is Noetherian, it is deduced that the Schur multiplier $H_2(G)$ is finitely generated
if and only if  $\pi_2^S(K(G, 1))$ is finitely generated.
\end{remark}

\section{Connection with Schur's theorem}\label{S:schur}

Given a group $G$, let $Z(G)$ denote its center. A famous theorem of Schur asserts that if $[G : Z(G)]$ is finite,
then $[G, G]$ is also finite (\!\cite{RoDe, RoJo}). We refer the reader to the recent, very nice survey article \cite{DKP} to see generalizations of Schur's theorem.
An equivalent statement is that $[G, G]$ is finite for any central extension  $1\to N \to G \to H \to 1$ of a finite group $H$.
Here, we present further generalizations of  Schur's theorem; it is also shown that the finiteness of the Schur multiplier of Noetherian groups is equivalent
to the analogue of Schur's theorem for Noetherian groups.

\begin{proposition} \label{prop}
Let $1\to N \to G \to H \to 1$ be a central extension of groups. Then, there exists an epimorphism $H\otimes H \to [G, G]$ such that
the following diagram commutes:
\[
\xymatrix{
H\otimes H \ar[r]^{1_{H\otimes H}}\ar[d] & H\otimes H \ar[d]^{}\\
[G, G] \ar[r]^{}& [H, H]} ,
\]
being the right vertical homomorphism  defined by $h\otimes h'\mapsto [h, h']$ for all $h, h'\in H$.
\end{proposition}
\begin{proof}
Starting from the following exact sequence of groups:
\[
N\otimes G \times G\otimes N \to G\otimes G \to H\otimes H \to 1,
\]
we denote the image of $N\otimes G \times G\otimes N$ into $G\otimes G$ by $X$. Then, the groups $\frac{G\otimes G} { X}$ and  $H\otimes H$ are isomorphic; let us denote the isomorphism by $\theta$, and
let $\lam \colon G\otimes G\to G$ be the homomorphism given by
$g\otimes g'\mapsto [g, g']$ for all $g, g'\in G$. Since $N$ is contained in the center of $G$, $\lam (X)=1$. Thus, $\lambda$ induces a well-defined homomorphism
$\lam^* \colon \frac{G\otimes G }{ X }\to G$ with image $[G, G]$.
Therefore, the homomorphism $\lam ^* \theta ^{-1} \colon  H\otimes H \to [G, G]$ satisfies the required conditions.
\end{proof}

Using this proposition, generalizations of Schur's theorem for various classes of groups can be proved. In the next result, the class of super-solvable groups
might be the only one for which the statement is new, according to the authors' knowledge.

\begin{proposition}
Let $1\to N \to G \to H \to 1$ be a central extension of groups. If $H$ belongs to the class of finite,  polycyclic, polycyclic-by-finite,
super-solvable or finite $p$-groups, then $[G, G]$ belongs to the same class.
\end{proposition}
\begin{proof}   It is proved in \cite{DLT}, \cite{ Ell87f} and \cite{Mo07} that if $H$ belongs to the class of  finite,  polycyclic, polycyclic-by-finite,
super-solvable or finite $p$-groups, then $H\otimes H$ also belongs to the same class.  By Proposition~\ref{prop}, $[G, G]$ is a homomorphic image of $H\otimes H$,
hence $[G, G]$ is finite, polycyclic, polycyclic-by-finite, super-solvable or finite $p$-group, respectively.
\end{proof}

\begin{corollary}
Let $G$ be a group. If $G/Z(G)$ belongs to the class of finite, polycyclic, polycyclic-by-finite, super-solvable or finite $p$-groups,
then $[G, G]$ belongs to the same class.
\end{corollary}

The next two propositions show how the veracity of Schur's theorem for Noetherian groups is related to the ``finiteness'' of the Schur multiplier of the same group.

\begin{proposition}
Let  $1\to N \to G \to H \to 1$ be a central extension of a Noetherian group $H$. If the Schur multiplier of $H$ is finitely generated, then $[G, G]$ is Noetherian.
\end{proposition}

\begin{proof} Since the Schur multiplier of $H$ is finitely generated, Proposition~\ref{paper55} implies that $H\otimes H$ is Noetherian. Therefore,
by Proposition~\ref{prop}, $[G, G]$ is also Noetherian.
\end{proof}

\begin{proposition}
Let $H$ be a Noetherian group. Suppose that for any central extension $1\to N \to G \to [H, H] \to 1$ of the derived subgroup $[H, H]$, $[G, G]$ is Noetherian.
Then the Schur multiplier of $H$ is finitely generated.
\end{proposition}
\begin{proof}
By Proposition~\ref{paper55}, it suffices to show that $H\otimes H$ is Noetherian. Consider the following central extension of groups:
\[
1\to \Ker \lam \to H\otimes H \overset{\lam }{\longrightarrow} [H, H]\to 1 ,
\]
where $\lam \colon H\otimes H \to [H, H]$ is given by $h\otimes h'\mapsto [h, h']$ for each $h, h'\in H$. By hypothesis, $[H\otimes H, H\otimes H]$ is Noetherian. Moreover,
taking into account Proposition~\ref{paper5}, it is clear that the abelian group $H\otimes H / [H\otimes H, H\otimes H]$ is finitely generated and hence Noetherian. Thus, from the following exact sequence
\[
1\to  [H\otimes H, H\otimes H] \to H\otimes H \to H\otimes H / [H\otimes H, H\otimes H]\to 1 ,
\]
we deduce that $H\otimes H$ is Noetherian.
\end{proof}

\begin{corollary}
Let $H$ be a perfect Noetherian group. Then the Schur multiplier of $H$ is finitely generated if and only if for any central extension $1\to N \to G \to H \to 1$ of $H$, $[G, G]$ is Noetherian.
\end{corollary}

\begin{corollary}
There exists a Noetherian group which does not have finitely generated Schur multiplier.
\end{corollary}
\begin{proof}
Ol'shanski\u\i \ showed that there exists a perfect group $G$ such that $Z(G)$ is a free abelian group of countable
rank, and $G/Z(G)$ is a Tarski monster (\!\cite{Olsh}). Thus, $[G, G]$ is not Noetherian.
\end{proof}

The next proposition provides a generalization of  Schur's theorem for finitely generated groups.

\begin{proposition} \label{prop2}
Let $H$ be a finitely generated group. Then, for any central extension $1\to N \to G \to H \to 1$ of groups, $[G, G]$
is finitely generated if and only if $[H, H]$ is finitely generated.
\end{proposition}

\begin{proof}
Firstly, it is clear that $[H, H]$ has to be finitely generated in order that $[G, G]$ is also finitely generated.
 Thus, we focus on proving that $[G, G]$ is finitely generated if so are both $H$ and $[H, H]$.
By  Proposition~\ref{paper5}, $H\otimes H$ is a finitely generated group; therefore, Proposition~\ref{prop} concludes the argument.
\end{proof}

\begin{corollary}
Let $G$ be a group such that $G/Z(G)$ is finitely generated. Then, $[G, G]$ is finitely generated if and only if the commutator subgroup of $G/Z(G)$ is too.
\end{corollary}

\begin{remark}
If $H$ is a group such that $[H, H]$ is finitely generated but $H$ itself is not, it holds that for any central extension
$1\to N \to G \to H \to 1$ of groups, $[G, G]$ need not be finitely generated. For instance, suppose that $G=F/[F, [F, F]]$ and $H=F/[F, F]$, where
$F$ is a free group with a basis of infinite rank. Then, $[G, G]$ is not finitely generated. This trivial example highlights that in Proposition~\ref{prop2}
we essentially need both $H$ and $[H, H]$ to be finitely generated.
\end{remark}

Given a group $G$, let $\Gamma_1(G), \Gamma_2(G), \Gamma_3(G) \dots,$ denote the \emph{derived series} of $G$, i.e.
$\Gamma_1(G)=G$ and $\Gamma_{i+1}(G)=[\Gamma_i(G), \Gamma_i(G)]$, for $i\geq 1$. Proposition~\ref{prop2} yields the following results.

\begin{corollary}
Let $H$ be a finitely generated group and $1\to N \to G \to H \to 1$ be a central extension of $H$. Then the following two statements are equivalent:
\begin{itemize}
  \item[(i)] $\Gamma_2(H), \Gamma_3(H), \dots$ is a sequence of finitely generated groups;
  \item[(ii)] $\Gamma_2(G), \Gamma_3(G), \dots$ is a sequence of finitely generated groups.
\end{itemize}
\end{corollary}

\begin{corollary}
Let $H$ be a Noetherian group and $1\to N \to G \to H \to 1$ be a central extension of $H$. Then $\Gamma_n(G)$ is finitely generated for $n\geq 2$.
\end{corollary}

Aiming to obtain the same results for the lower central series, we present the following proposition. Recall that the \emph{lower central series} $\gamma_1(G), \gamma_2(G), \gamma_3(G) \dots,$
of a group $G$ is defined by $\gamma_1(G)=G$ and $\gamma_{i+1}(G)=[G, \gamma_i(G)]$, for $i\geq 1$.

\begin{proposition} \label{prop3}
Let $H$ be a finitely generated group and $1\to N \to G \to H \to 1$ be a central extension of $H$. Fix an integer $n\geq 1$ and suppose that $\gamma_n(G)$ is a finitely
generated group. Then, $\gamma_{n+1}(G)$ is finitely generated if and only if $\gamma_{n+1}(H)$ is finitely generated.
\end{proposition}
\begin{proof} The idea of this demonstration is to show that if  $\gamma_{n+1}(H)$ is finitely generated, then $\gamma_{n+1}(G)$ is finitely generated too. We have the following extensions
of groups:
\[
1\to N \to G \to H \to 1 ,
\]
and
\[
1\to N_n \to \gamma_n(G) \to \gamma_n(H) \to 1,
\]
where $N_n = N\cap \gamma_n(G)$. Therefore, considering the non-abelian tensor products we get the following exact sequence:
\[
N\otimes \gamma_n(G) \times N_n\otimes G \to G\otimes \gamma_n(G) \to H\otimes \gamma_n(H) \to 1 .
\]
This sequence leads to the isomorphism:
\[
\frac{ G\otimes \gamma_n(G)}{X} \cong  H\otimes \gamma_n(H),
\]
where $X$ is the image of $N\otimes \gamma_n(G) \times N_n\otimes G$ into $G\otimes \gamma_n(G)$. Now, let
$\de \colon G\otimes \gamma_n(G)\to \gamma_{n+1}(G)$ be the homomorphism given by $g\otimes g' \mapsto [g, g']$ for
each $g\in G$ and $g' \in \gamma_{n}(G)$. Since $N$ and $N_n$ are contained in the center of $G$, it holds that $\de (X)=1$. Therefore, $\de$ induces an epimorphism
from $ H\otimes \gamma_n(H)$ to $\gamma_{n+1}(G)$, and it suffices to show that $H\otimes \gamma_n (H)$ is finitely generated. Note that
$D_{H}(\gamma_n(H))=\gamma_{n+1}(H)$ and $D_{\gamma_n (H)}(H)=\gamma_{n+1}(H)$. By Proposition~\ref{paper5}, $H\otimes \gamma_n (H)$ is
finitely generated. Hence, $\gamma_{n+1}(G)$ is finitely generated too.
\end{proof}

Taking into account the previous proposition, it is straightforward to derive the following results.

\begin{corollary}
Let $H$ be a finitely generated group and $1\to N \to G \to H \to 1$ be a central extension of $H$. Then the following two statements are equivalent:
\begin{itemize}
  \item[(i)] $\gamma_2(H), \gamma_3(H), \dots$ is a sequence of finitely generated groups;
  \item[(ii)] $\gamma_2(G), \gamma_3(G), \dots$ is a sequence of finitely generated groups.
\end{itemize}
\end{corollary}

\begin{corollary}
Let $H$ be a Noetherian group and $1\to N \to G \to H \to 1$ be a central extension of $H$. Then $\gamma_n(G)$ is finitely generated for $n\geq 2$.
\end{corollary}

\

\section*{Acknowledgments}
The authors were supported by  Agencia Estatal de Investigaci\'on (Spain), grant MTM2016-79661-P (European FEDER support included, UE).
 The first author was also partially supported by  visiting
program (USC).
P. P\'aez-Guill\'an was also supported by FPU scholarship, Ministerio de Educaci\'on, Cultura
y Deporte (Spain).

\end{document}